\documentclass[10pt]{article}
\usepackage{amssymb,amsmath,amsthm}
\usepackage{color}
\usepackage{fullpage}
\usepackage{graphicx}

\newtheorem{theorem}{Theorem}
\newtheorem{lemma}[theorem]{Lemma}
\newtheorem{claim}[theorem]{Claim}

\title{The digrundy number of digraphs}
\author{Gabriela Araujo-Pardo\footnotemark[2] \and Juan Jos{\' e} Montellano-Ballesteros\footnotemark[2] \and Mika Olsen \footnotemark[3] \and Christian Rubio-Montiel\footnotemark[4]}

\begin{document}
\maketitle

\def\thefootnote{\fnsymbol{footnote}}
\footnotetext[2]{Instituto de Matem{\'a}ticas, Universidad Nacional Aut{\'o}noma de M{\' e}xico, Mexico City, Mexico. {\tt [garaujo|juancho]@math.unam.mx}.}
\footnotetext[3]{Departamento de Matem{\' a}ticas Aplicadas y Sistemas, Universidad Aut\'onoma Metropolitana - Cuajimalpa, Mexico City, Mexico. {\tt olsen.mika@cua.uam.mx}.}
\footnotetext[4]{Divisi{\' o}n de Matem{\' a}ticas e Ingenier{\'i}a, FES Acatl{\' a}n, Universidad Nacional Aut{\'o}noma de M{\'e}xico, Naucalpan, Mexico. {\tt christian.rubio@acatlan.unam.mx}.}

\begin{abstract} 
We extend the Grundy number and the ochromatic number,  parameters on graph colorings, to digraph colorings, we call them {\emph{digrundy number}} and {\emph{diochromatic number}}, respectively. First, we  prove that for every digraph the diochromatic number equals the digrundy number (as it happen for graphs). Then, we prove the interpolation property and the Nordhaus-Gaddum relations for the digrundy number,  and improve the Nordhaus-Gaddum relations for the dichromatic and diachromatic numbers bounded previously by the  authors in [Electron. J. Combin. 25 (2018) no. 3, Paper {\#} 3.51, 17 pp.]
\end{abstract}


\section{Introduction}

It is common that classical results or problems on graph theory provide us interesting questions on digraph theory. An interesting question is what is the natural generalization of the chromatic number in the class of digraphs. In 1982 Neumann-Lara introduced the concept of dichromatic number as a generalization of the chromatic number to the class of digraphs. Specifically, in \cite{MR3000989, MR3322691, MR3979228, MR3692144, MR3705781, MR3593495} the authors study the dichromatic number in order to extend results on the chromatic number of graphs to the class of digraphs.
Furthermore, as an anecdote, M. Skoviera\footnote{Oral communication.},  after a  talk about the diachromatic number  in a conference,   said: ``It looks that dichromatic number is the correct generalization for the chromatic number'', confirming the intuitiveness of the dichromatic number as a generalization of the chromatic number. 

We consider finite digraphs, without loops and symmetric arcs are permitted. A \emph{(vertex) coloring} of a digraph $D$ is \emph{acyclic} if the induced subgraph of each chromatic class is acyclic, i.e., it only admits no directed cycles. The \emph{dichromatic number} $dc(D)$ of $D$ is the smallest $k$ such that $D$ has an acyclic coloring with $k$ colors \cite{MR693366}. This parameter is a generalization of the chromatic number for graphs, see \cite{MR2564801,MR1133813,MR3692144,MR3705781,MR2781992,MR3711038,MR1817491,MR3112565} for old and new results about dichromatic number. For a detailed introduction to digraphs we refer to \cite{MR2472389}.
A coloring of a digraph $D$ is \emph{complete} if for every pair $(i,j)$ of different colors there is at least one arc $(u,v)$ such that $u$ is colored $i$ and $v$ is colored $j$ \cite{MR2998438}. Note that any acyclic coloring of $D$ with $dc(D)$ colors is a complete coloring. The \emph{diachromatic number} $dac(D)$ of a digraph $D$ is the largest number of colors for which there exists an acyclic and complete coloring of $D$. Hence, the dichromatic and diachromatic numbers of a digraph $D$ are, respectively, the smallest and the largest number of colors in a complete acyclic coloring of $D$, see \cite{MR3875016}.

Let $D$ be a digraph of order $n$ whose vertices are listed in some specified order. In a \emph{greedy coloring} of $D$, the vertices are successively colored with positive integers according to an algorithm that assigns to the vertex under consideration the smallest available color. Hence, if the vertices of $D$ are listed in the order $v_1,v_2,\dots,v_n$, then the resulting greedy coloring $\varsigma$ assigns the color $1$ to $v_1$, that is, $\varsigma(v_1)=1$. If $v_1$ and $v_2$ are not a $2$-cycle, then assign $\varsigma(v_2)=1$, else $\varsigma(v_2)=2$. In general, suppose that the first $j$ vertices $v_1,v_2,\dots,v_j$, where $1\leq j < n$, in the sequence have been colored with the colors $1,\dots ,{t-1}$. Let $\{V_i\}_{i=1}^{t-1}$ be the set of chromatic classes.
Consider the vertex $v_{j+1}$, if there exists a chromatic class $V_i$, with the smallest $i$, for which $V_i\cup \{v_{j+1}\}$ is acyclic, then $\varsigma(v_{j+1})=i$, else $\varsigma(v_{j+1})=t$. When the algorithm ends, the vertices of $D$ have been assigned colors from the set $[k]\colon =\{1,2,\dots,k\}$ for some positive integer $k$. Note that any greedy coloring is a complete coloring. The \emph{digrundy number} $dG(D)$ is the largest number of colors in a greedy coloring, see \cite{MR3875016}. 

In this paper, we explore the analogue parameter to the Grundy number for digraphs which we call the digrundy number (we recall that the Grundy number $\Gamma$ is also known as the First-Fit number $\chi_{FF}$). 
The paper is organized as follows: In Section \ref{sec2} we prove the interpolation theorem for digrundy number and give a characterization of  digrundy number and in Section \ref{sec4} we prove the inequalities called the Nordhaus-Gaddum relations for digrundy number and we improve those relations for dichromatic and  diachromatic numbers.

\section{The digrundy and diochromatic numbers}\label{sec2}
Since the digrundy number $dG(D)$ is the largest number of colors in a greedy coloring, it follows that:
$$dc(D)\leq dG(D)\leq dac(D).$$

A \emph{digrundy coloring} of a graph $D$ is an acyclic coloring of $D$ having the property that for every two colors $i$ and $j$ with $i<j$, every vertex colored $j$ has a neighbor colored $i$. It is not hard to see that a coloring $\varsigma$ of a digraph $D$ is a digrundy coloring of $D$ if and only if $\varsigma$ is a greedy coloring of $D$. Therefore, for each vertex $v$ in the chromatic class $j$ and each chromatic class $V_i$, with $i<j$, $N^+(v)\cap V_i\not=\emptyset$ and $N^-(v)\cap V_i\not=\emptyset$ then,
\[dG(D)\leq \min\{\Delta^+(D),\Delta^-(D)\}+1.\]
Moreover, we have the following remark.
\begin{lemma}\label{Lem dc greedy}
 For each digraph $D$, there exists an ordering $\phi$ of the vertices of $D$ such that the digrundy coloring attains the dichromatic number of $D$.  
\end{lemma}
\begin{proof}
 Let $D$ be a digraph and consider an acyclic coloring $\varphi:V(D)\to [dc(D)]$. Consider an ordering of $V(G)$ which respects the order of the chromatic classes. Recall that the ordering in each chromatic class is irrelevant since each class is acyclic.  If $\varphi$ is not greedy, then let $i$ be the greatest integer such that $\varphi$ is a greedy coloring restricted to $V_1\cup V_2\cup\dots\cup V_i$. Denote by  $\varphi_{i}$ the greedy coloring of $V_1\cup V_2\cup\dots\cup V_i$.
 Let $\varphi_{i+1}:V_1\cup V_2\cup\dots\cup V_{i+1}\to[i+1]$ be the acyclic coloring such that $\varphi_{i+1}(u)=\varphi_{i}(u)$ for $u\in V_1\cup V_2\cup\dots\cup V_i$ and $\varphi_{i+1}$ recolors the vertices of $V_{i+1}$, respecting the order in $V_{i+1}$, using the greedy coloring. The coloring $\varphi$ is an optimal coloring of $D$, thus, when we recolor the vertices of $V_{i+1}$, there must be some vertex of color $i+1$ and since $V_{i+1}$ is acyclic, $\varphi_{i+1}$ uses exactly $i+1$ colors. Applying this process to each chromatic class $V_j$, with $j>i$ we obtain a greedy coloring of $D$ using $dc(D)$ colors. 
\end{proof}

In \cite{MR3875016}, the authors  proved the interpolation theorem for diachromatic number of a digraph $D$, that is, for every $k$ such that $dc(D) \leq k \leq dac(D)$ there exists an acyclic and complete coloring of $D$ using $k$ colors. In this section, we prove the interpolation theorem for  digrundy number. The version for graphs was proved in \cite{MR539075}, for further information see \cite{MR2450569}.

\begin{theorem}
For a digraph $D$ and an integer $k$ with $dc(D) \leq k \leq dG(D)$, there is a digrundy coloring of $D$ using $k$ colors.
\end{theorem}
\begin{proof}
Let $\varsigma$ be a digrundy coloring of $D$ using the set of colors $[dG(D)]$, let $\phi$ be the corresponding vertex ordering  and let $V_1,V_2,\dots,V_{dG(D)}$ be the color classes of $\varsigma$, where $V_i$ consists of the vertices colored $i$ by $\varsigma$ for $i\in [dG(D)]	$. For each integer $i$ with $1\leq i \leq dG(D)+1$, let $a_i$ be the smallest number of colors in an acyclic coloring of $D$ which coincides with $\varsigma$ for each vertex belonging to $V_1\cup V_2\cup\dots\cup V_{i-1}$. Observe that  $a_{dG(D)+1} = dG(D)$. Furthermore, for each integer $i$ with $1\leq i \leq dG(D)$, let $D_i$ be the subgraph of $D$ induced by $V_i\cup V_{i+1}\cup\dots\cup V_{dG(D)}$. Suppose that $a_1=dc(D)$. Since each vertex $x$ in $V_i\cup V_{i+1}\cup\dots\cup V_{dG(D)}$ is in at least one directed cycle with the others vertices 
in each of the color classes $V_1,V_2,\dots, V_{i-1}$, it follows that in every coloring of $D$ that coincides with $\varsigma$ on $V_1\cup V_2\cup\dots\cup V_{i-1}$, none of the colors $1,2,\dots, i-1$ can be used for a vertex of $D_i$ and so 
\begin{equation} \label{Eq1}
a_i = (i-1) + dc(D_i).
\end{equation}
Since $D_{i+1}$ is a subgraph of $D_i$, it follows that $dc(D_{i+1}) \leq dc(D_i)$. Furthermore, a coloring of $D_i$ using $dc(D_i)$ can be obtained from a coloring of $D_{i+1}$ using $dc(D_{i+1})$ by assigning all of the vertices in $V_i$ the same color but one that is different from the colors used in the coloring of $D_{i+1}$ using $dc(D_{i+1})$. Thus
\begin{equation} \label{Eq2}
dc(D_i)-1\leq dc(D_{i+1})\leq dc(D_i).
\end{equation}
By Equations (\ref{Eq1}) and (\ref{Eq2}), 
\[a_i=(i-1)+dc(D_i)=i+(dc(D_i)-1)\leq i+dc(D_{i+1})\]
\[\leq i+dc(D_i)=1+(i-1)+dc(D_i)=1+a_i.\]
Therefore, $a_i\leq i+dc(D_{i+1})\leq 1+a_i$. Since $a_{i+1}=i+dc(D_{i+1})$, it follows that $$a_i\leq a_{i+1} \leq 1+a_i.$$ On the other hand, $a_1=dc(D)$ and $a_{dG(D)+1}=dG(D)$. Thus, for each integer $k$ with $dc(D)\leq k \leq dG(D)$, there is an integer $i$ with $1\leq i \leq dG(D)+1$ such that $a_i=k$.
By Lemma \ref{Lem dc greedy}, we may assume that $dc(D)<k<dG(D)$. Thus there exists a coloring $\varsigma'$ of $D$ using $k$ colors such that $\varsigma'$ coincides with $\varsigma$ for each vertex belonging to $V_1\cup V_2 \cup \dots \cup V_{i-1}$. 

By Lemma \ref{Lem dc greedy}, let $\phi''$ be a vertex ordering such that $\phi$ and $\phi''$ coincides for $v\in V_1\cup V_2\cup\dots\cup V_{i-1}$ and such that  when we apply the greedy algorithm on $D_i$, we obtain $dc(D_i)$ colors.  

Let $\varsigma''$ be the greedy coloring with respect to $\phi''$. Suppose that $\varsigma''$ is an coloring of $D$ using $l$ colors. Then $\varsigma''$ is a digrundy coloring of $D$ using $l$ colors such that $\varsigma''$ coincides with $\varsigma'$ and $\varsigma$ on all of the vertices in $V_1\cup V_2\cup \dots \cup V_{i-1}$ and $\varsigma''$ assigns to each vertex of $D$ a color not greater than the color assigned to the vertex by $\varsigma'$. Therefore, $l\leq k$. On the other hand, by the definition of $a_i$, the coloring $\varsigma''$ cannot use less than $k=a_i$ colors, which implies that $l=k$ and so $\varsigma''$ is digrundy coloring of $D$ using $k$ colors.
\end{proof}


In 1982 G. Simmons \cite{MR726050} introduced a new type of coloring of a graph $G$ based on orderings of the vertices of $G$, which is similar to but not identical to greedy colorings of $G$. We extend this definition to digraphs using acyclic colorings.

Let $\phi\colon v_1,v_2,\dots,v_n$ be an ordering of the vertices of a digraph $D$. An acyclic coloring $\varsigma\colon V(D)\rightarrow \mathbb{N}$ of $D$ is a \emph{parsimonious $\phi$-coloring} of $D$ if the vertices of $D$ are colored in the order $\phi$, beginning with $\varsigma(v_1)=1$, such that each vertex $v_{i+1}$ $(1\leq i\leq n-1)$ must be assigned a color that has been used to color one or more of the vertices $v_1,v_2,\dots,v_i$ if possible. If $v_{i+1}$ can be assigned more than one color, then a color must be selected that results in using the fewest number of colors needed to color $D$. If $v_{i+1}$ form a directed cycle to every currently chromatic class, then $\varsigma (v_{i+1})$ is defined as the smallest positive integer not yet used. The \emph{parsimonious $\phi$-coloring number} $dc_\phi (D)$ of $D$ is the minimum number of colors in a parsimonious $\phi$-coloring of $D$. The maximum value of $dc_\phi (D)$ over all orderings $\phi$ of the vertices of $D$ is the \emph{ordered dichromatic number} or, more simply, the \emph{diochromatic number} of $D$, which is denoted by $dc^o (D)$.

P. Erd{\H o}s, W. Hare, S. Hedetniemi, and R. Laskar \cite{MR889347} showed that the ochromatic number of every graph always equals its Grundy number. This is also true for these generalizations for digraphs.

\begin{theorem}
For every digraph $D$, $dG(D)=dc^o(D)$.
\end{theorem}
\begin{proof}
In order to show that $dc^o(D)\leq dG(D)$, let $\phi\colon v_1,v_2,\dots,v_n$ be an ordering of the vertices of $D$ such that $dc_\phi (D)=dc^o(D)$. Consider the parsimonious coloring using $dc^o(D)$ colors, obtained using a greedy coloring, that is, whenever there is a choice of a color for a vertex, the smallest possible color is chosen. Suppose that this results in an coloring using $l$ colors of $D$. Then $dc_\phi (D)\leq l$. Furthermore, this coloring using $l$ colors is a Grundy coloring using $l$ colors. Therefore, $dG(D)\geq l$ and so
\[dc^o(D)=dc_\phi (D) \leq l \leq dG(D),\]
producing the desired inequality.

We show that $dc^o(D)\geq dG(D)$. Let $dG(D)=k$. Consider a Grundy coloring of the vertices of $D$, using the colors $1,2,\dots,k$, and let $V_1,V_2,\dots,V_k$ denote the chromatic classes $(1\leq i \leq k)$. Let $\phi\colon v_1,v_2,\dots,v_n$ be any ordering of $D$ in which the vertices of $V_1$ are listed first in some order, the vertices of $V_2$ are listed next in some order, and so on until finally listing the vertices of $V_k$ in some order. We now compute $dc_\phi (D)$. Assign $v_1$ the color $1$. Since $V_1$ is acyclic, every vertex in $\phi$ that belongs to $V_1$ is not in a monochromatic directed cycle using only vertices of $V_1$, therefore every vertex in $V_1$ must be colored 1 as well. Assume, for an integer $r$ with $1\leq r < k$, that the parsimonious coloring has assigned the color $i$ to every vertex in $V_i$ for $1\leq i \leq r$. Now, consider the vertices in $\phi$ that belong to $V_{r+1}$. Let $v_a$ be the first vertex appearing in $\phi$ that belongs to $V_{r+1}$. Since $v_a$ is in a directed cycle for each $V_i$ for every $i$ with $1\leq i \leq r$, it follows that $v_a$ cannot be colored any of the colors $1,2,\dots ,r$. Hence, the new color $r+1$ is assigned to $v_a$. Now if $v_b$ is any vertex belonging to $V_{r+1}$ such that $b>a$, then $v_b$ cannot be colored any of the colors $1,2,\dots,r$ since $v_b$ is in a directed cycle for each $V_i$ for $1\leq i \leq r$. However, since $v_b$ is not in a directed cycle for $V_{r+1}$, it follows that $v_b$ must be colored $r+1$. By mathematical induction, $dc_\phi (D) = k$. Thus, $dc^o (D)\geq dG(D)$, and the result follows.

\end{proof}


\section{On the Nordhaus-Gaddum relations}\label{sec4}

The Nordhaus-Gaddum inequality \cite{MR0078685} states that for every graph $G$ of order $n$\[\chi(G)+\chi(G^{c})\leq n+1.\]
These relations were extended to the pseudoachromatic numbers \cite{MR0256930} getting that for every graph $G$ of order $n$ \[\alpha(G)+\chi(G^{c})\leq n+1\qquad \text{ and }\qquad \alpha(G)+\alpha(G^{c})\leq\psi(G)+\psi(G^{c})\leq \left\lceil \frac{4n}{3}\right\rceil. \]
 And for the Grundy number \cite{MR2432888} for every graph $G$ of order $n\geq 10$
\[\Gamma(G)+\Gamma(G^{c})\leq \left\lfloor \frac{5n+2}{4}\right\rfloor . \]
For digraphs, there exists the following results \cite{MR3875016}.

If $D$ is a digraph of order $n$, then 
\begin{equation}\label{Eq3}
dc(D)+dc(D^{c})\leq \left\lceil \frac{4n}{3}\right\rceil\qquad \text{ and }\qquad dac(D)+dac(D^{c})\leq \left\lceil \frac{3n}{2}\right\rceil .
\end{equation}

In this section, we improve the upper bounds of Equation \ref{Eq3} and we prove a similar result for the digrundy number.

\begin{theorem}
If $D$ is a digraph of order $n$, then
$dc(D)+dc(D^{c})\leq n+1$.
\end{theorem}
\begin{proof}
The proof is by induction on $n$. The case of $D=K_1$ is trivial. Suppose that for each digraph $F$ of order at most $n-1\geq 1$, $dc(F)+dc(F^{c})\leq n$.
Let $D$ be a digraph of order $n$ and let $x\in V(D)$. Take an acyclic and complete coloring using $k$ colors of $D-x$ for $k=dc(D-x)$. Therefore, $dc(D)$ is at most $dc(D-x)+1$. Similarly, $dc(D^c)\leq dc(D^c-x)+1$. Hence, by induction hypothesis,
\[dc(D)+dc(D^c)-2\leq dc(D-x)+dc(D^c-x)\leq n\]
and  $dc(D)+dc(D^c)\leq n+2.$

Supose that $dc(D)<dc(D-x)+1$ or $dc(D^c)<dc(D^c-x)+1$, then 
\[dc(D)+dc(D^c)\leq n+1.\]
Assume that $dc(D)=dc(D-x)+1$ and $dc(D^c)=dc(D^c-x)+1$, which means that for each chromatic class $X$ of $D$ and $D^c$, $X\cup \{x\}$ contains a cycle. Then $2dc(D)\leq d_D(x)=d^+_D(x)+d^-_D(x)$ and $2dc(D^c)\leq d_{D^c}(x)=d^+_{D^c}(x)+d^-_{D^c}(x)$ then
\[2dc(D)+2dc(D^c)\leq d_D(x)+d_{D^c}(x)=2(n-1)\]
and the result follows.
\end{proof}

\begin{theorem}\label{dac}
If $D$ is a digraph of order $n$, then
\[dac(D)+dac(D^{c})\leq \left\lceil \frac{4n}{3}\right\rceil.\]
\end{theorem}
\begin{proof}
Let $dac(D)+dac(D^{c})=x$. Without loss of generality, $\frac{x}{2}\leq dac(D)$, that is, $dac(D)=\frac{x}{2}+\delta$ for some $0\leq\delta\leq\frac{x}{2}$. Let  $\omega$ denote the maximum order of a complete subdigraph (a complete symmetric digraph) in $D$. Since the set of singular chromatic classes induces a complete subdigraph in $D$ it follows that
\[dac(D)=\frac{x}{2}+\delta\leq \omega+\frac{n-\omega}{2}\] 
thus $x+2\delta\leq 2\omega+n-\omega$ and $x-n+2\delta\leq \omega.$

On the other hand, $dac(D^c)\leq n-\omega+1$ because each complete subdigraph of $D$ is an independent set of vertices in  $D^c$. Hence \[dac(D^c)\leq n+1-x+n-2\delta=2n+1-x-2\delta,\]
 \[x=dac(D)+dac(D^{c})\leq \frac{x}{2}+\delta + 2n+1-x-2\delta=-\frac{x}{2}-\delta + 2n+1\]
and $\frac{3x}{2}\leq 2n+1-\delta.$ Finally, $dac(D)+dac(D^{c})\leq \left\lfloor\frac{4n+2}{3}\right\rfloor$ and the result follows.
\end{proof}

Finally, we prove the Nordhaus-Gaddum for the digrundy number.

\begin{theorem}
If $D$ is a digraph of order $n$, then \[dG(D)+dG(D^{c})\leq\begin{cases}
\begin{array}{c}
n+1\\
n+2\\
12\\
\left\lfloor \frac{5n+2}{4}\right\rfloor 
\end{array} & \begin{array}{c}
\textrm{ if }n\leq4;\\
\textrm{ if }n\leq8;\\
\textrm{ if }n=9;\\
\textrm{ if }n\geq10.
\end{array}\end{cases}\]
\end{theorem}
\begin{proof}

Let $\mathcal{A}=\{A_1,\dots,A_p\}$ and $\mathcal{B}=\{B_1,\dots,B_q\}$ be optimal ordered vertex partitions of $D$ and $D^c$ for a digrundy coloring, respectively. Suppose that $\mathcal{A}$ has $a_1$ sets of order one, $a_2$ sets of order two and $a_3$ sets of order at least three. Similarly, $\mathcal{B}$ has $b_1$ sets of order one, $b_2$ sets of order two and $b_3$ sets of order at least three. From the assumption, $dG(D)=a_1+a_2+a_3$, $dG(D^c)=b_1+b_2+b_3$ and the definitions of $a_i$ and $b_i$ we have $a_1+2a_2+3a_3\leq n$ and $b_1+2b_2+3b_3\leq n$. We can write 
\begin{equation}\label{eq4}
a_1+2a_2+3a_3+\epsilon_a= n
\end{equation}
and
\begin{equation}\label{eq5}
b_1+2b_2+3b_3+\epsilon_b= n,
\end{equation}
where $\epsilon_a,\epsilon_b\geq 0$ are the excess.

Consider the sets of order one of $\mathcal{A}$ and $\mathcal{B}$. We may suppose (eventually reorder) that they come last in the orderings. 
Since $K$ and $L$ contains the singular classes of $D$ and $D^c$ respectively,  where $K=\{v\in A_i\colon |A_i|=1\}$ spans a complete subdigraph in $D$ and $L=\{v\in B_j\in \mathcal{B}\colon |B_j|=1\}$ spans an independent set in $D$, thus $|K\cap L|\leq 1$. If $|K\cap L|=1$, then $dG(D)+dG(D^c)\leq n+1$  because if $\{x\}=K\cap L $, then $2(|\mathcal{A}|-1)\leq d^+_D(x)+d^-_D(x)$ and $2(|\mathcal{B}|-1)\leq d^+_{D^c}(x)+d^-_{D^c}(x)$ and $dG(D)+dG(D^c)\le |\mathcal{A}|+|\mathcal{B}|\leq  n+1<\frac{5n+5}{4}$. 

Assume $K\cap L=\emptyset$, we prove that  $dG(D)+dG(D^c)\le |\mathcal{A}|+|\mathcal{B}|\le n+1<\frac{5n+5}{4}$. Let $\alpha_2$ and $\alpha_3$ be the number of sets in $\mathcal{A}$ contained in $L$ with  2- and at least 3-elements  respectively, $\alpha = \alpha_2+\alpha_3$, and define similarly $\beta_2$ and $\beta_3$ for $\mathcal{B}$. Since $L$ (respectively $K$) is an independent set in $D$ (respectively in $D^c$), it follows that
$$\alpha, \beta \leq 1.$$
Classify the $2$-element sets into three groups. There are $a_{2,t}$ of them meeting $L$ in exactly $t$ elements. Define $b_{2,t}$ analogously (i.e., the number of $2$-sets of $\mathcal{B}$ meeting $K$ in $t$ vertices). We have
\[a_{2,2}=\alpha_2,\qquad a_2=a_{2,0}+a_{2,1}+a_{2,2},\qquad b_{2,2}=\beta_2,\qquad b_2=b_{2,0}+b_{2,1}+b_{2,2}.\]

All but $\alpha$ parts of $\mathcal{A}$ have points outside $L$, and at least $a_{2,0}$ of them have two or more. We get that $|\mathcal{A}|-\alpha + a_{2,0}\leq n-|L|$. Again, write this (and its analogue for $\mathcal{B}$, $|\mathcal{B}|-\beta + b_{2,0}\leq n-|K|$) in the following form:
\begin{equation}\label{eq6}
a_1 + a_2 + a_3 + a_{2,0} + b_1 = n + \alpha - \epsilon_\alpha
\end{equation}
\begin{equation}\label{eq7}
b_1 + b_2 + b_3 + b_{2,0} + a_1 = n + \beta - \epsilon_\beta
\end{equation}

Consider an $a_{2,1}$ two-element $\mathcal{A}$-set, say $\{v,v'\}$, that intersect $L$ in exactly one vertex, say $v\in L$ and $v'\not\in L$. Denote the set of these vertices $v\in L$ by $L_1$, and the set of vertices $v'\not\in L$ by $S$. Similarly, $K_1\colon = \{u \in K\colon \exists u' \not\in K$ such that $\{u,u'\}\in \mathcal{B}\}$, and $T\colon =\{u' \not\in K\colon \exists u \in K$ such that $\{u,u'\}\in \mathcal{B}\}$. We have
\[|S|=a_{2,1},\qquad S\cap (K \cup L) = \emptyset,\qquad |T| = b_{2,1},\qquad T\cap (K \cup L) = \emptyset.\]

\begin{claim}\label{claim6} $|S\cap T |\leq 1.$
\begin{proof}
The sets of order one of an optimal ordered partition can be taken such that they have the greatest color labels, otherwise, we can reorder them in such a way.

Assume, on the contrary, that $x_1,x_2\in S \cap T$. This means that there are $u_1,u_2\in L$ such that the two-element parts $\{u_1,x_1\}$ and $\{u_2,x_2\}$ belong to $\mathcal{A}$, and there are $v_1,v_2\in K$ such that $\{v_1,x_1\}$ and $\{v_2,x_2\}$ belong to $\mathcal{B}$. By definition we already know the status of the pairs, namely $v_1v_2,v_2v_1\in F(D)$ and $u_1u_2,u_2u_1\notin F(D)$. Let $<_{\mathcal{A}}$ denote position of the elements in the  ordering $\mathcal{A}$. 
By symmetry (between $\{u_1,x_1\}$ and $\{u_2,x_2\}$), we may suppose that the order of these classes of the partition ${\mathcal{A}}$  is
\[\{u_1,x_1\}<_{\mathcal{A}}\{u_2,x_2\}<_{\mathcal{A}}\{v_1\}<_{\mathcal{A}}\{v_2\}.\]
Then $\{u_1,x_1\}$ and $u_2$ implies $x_1u_2,u_2x_1\in F(D)$, i.e., $x_1u_2,u_2x_1\notin F(D^c)$. Therefore, $\{v_1,x_1\}$ and $u_2$ implies $v_1u_2,u_2v_1\in F(D^c)$ since $\{v_1,x_1\}<_{\mathcal{B}} \{u_2\}$, see Figure \ref{fig1} a).

Note that, $\{u_2,x_2\}$ and $v_1$ implies $x_2v_1,v_1x_2\in F(D)$, i.e., $x_2v_1,v_1x_2\notin F(D^c)$. This implies that $\{v_1,x_1\}<_{\mathcal{B}}\{v_2,x_2\}$, otherwise, $v_1$ violate the greedy requirement of the partition $\mathcal{B}$. Then, $\{v_1,x_1\}$ and $v_2$ implies $x_1v_2,v_2x_1\in F(D^c)$ and $\{v_1,x_1\}$ and $x_2$ implies $x_1x_2,x_2x_1\in F(D^c)$, i.e., $x_1v_2,v_2x_1\notin F(D)$ and  $x_1x_2,x_2x_1\notin F(D)$, see Figure \ref{fig1} b).

Finally, $\{u_1,x_1\}$ and $x_2$ implies $x_2u_1,u_1x_2\in F(D)$, i.e., $x_2u_1,u_1x_2\notin F(D^c)$. On one hand, $\{u_1,x_1\}$ and $v_2$ implies $v_2u_1,u_1v_2\in F(D)$. On the other hand, $\{v_2,x_2\}$ and $u_1$ implies $v_2u_1,u_1v_2\in F(D^c)$ which is impossible, and the lemma follows, see Figure \ref{fig1} c).

\begin{figure}[ht!]
\begin{center}
\includegraphics{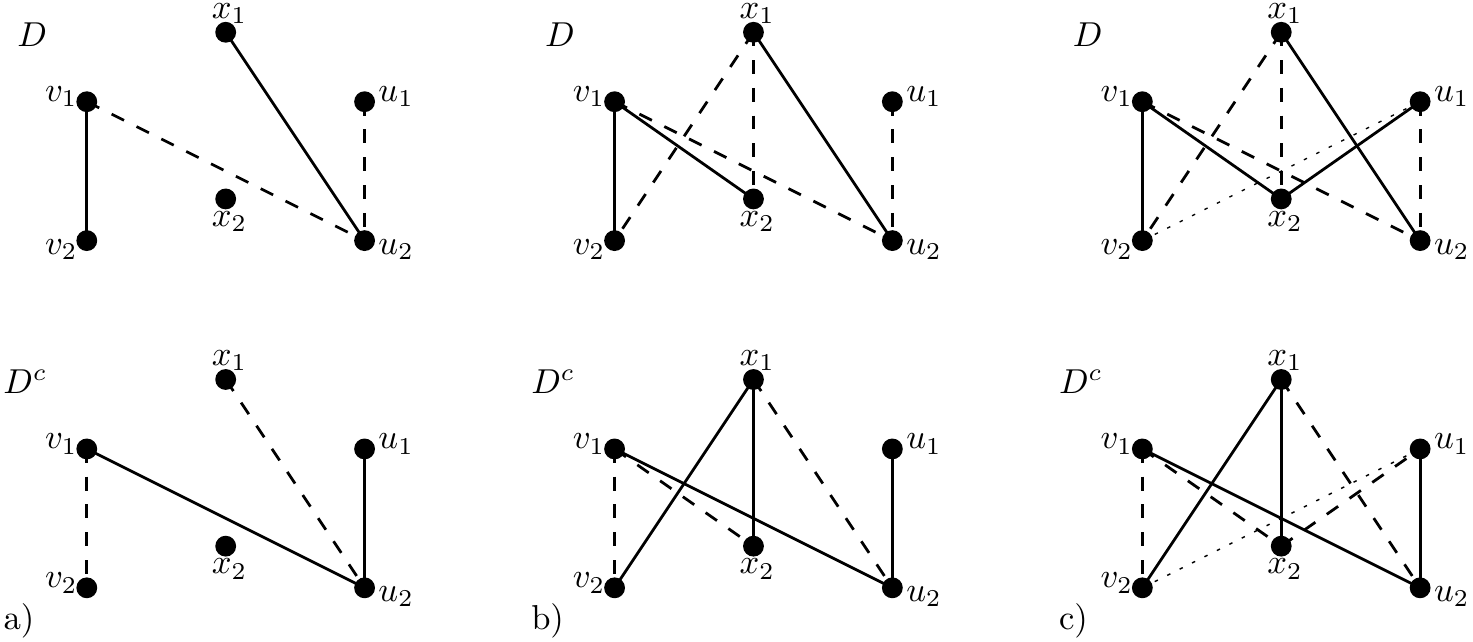}
\caption{Proof of Claim \ref{claim6}. Digons are represented with edges and dashed edges represent edges in the complement.}\label{fig1}
\end{center}
\end{figure}

\end{proof}
\end{claim}

Claim \ref{claim6} shows that the sets $K$, $L$, $S$, $T$ are almost disjoint. Let $\gamma = |S\cap T|$  and denote by $n-\epsilon_\gamma$ the order of the union of these four sets. By Claim \ref{claim6}, $\gamma \le1$. We obtain
\begin{equation}\label{eq8}
|L \cup K| + |S\cup T | = a_1 + b_1 + a_{2,1} + b_{2,1} - \gamma = n - \epsilon_\gamma.
\end{equation}

Adding the five equalities (\ref{eq4})-(\ref{eq8}) and denoting $\epsilon = \epsilon_a + \epsilon_b + \epsilon_\alpha + \epsilon_\beta + \epsilon_\gamma$ we get
\[4(a_1 + a_2 + a_3 + b_1 + b_2 + b_3 ) = 5n + (\alpha + \beta + \gamma) + (\alpha_2 + \beta_2 ) - \epsilon = 5n + s\]
That is, when $K\cap L=\emptyset$ we have that  $dG(D)+dG(D^c)\leq \frac{5n+s}{4}$ for some integer $s$. Since $\alpha, \beta, \gamma\le1$ it follows that $s\leq 5$.

In both cases, $K\cap L\neq\emptyset$ and $K\cap L=\emptyset$, $dG(D)+dG(D^c)\leq \frac{5n+s}{4}$ for some $s\le5$. The following claim is essential in order to prove that $s\le4$.

	
\begin{claim}\label{claim7} If $\alpha = 1$ then
\begin{enumerate}
\item[(1)] there is no class $B \in \mathcal{B}$ with $B \subset S$;
\item[(2)] there is no class $B \in \mathcal{B}$, $B\subset S \cup K$ with $|B \cap S| = |B|-1$;
\item[(3)] there is no class $A_i \in \mathcal{A}$, $A_i \subset L \cup T$ with $|A_i \cap T | = 1$.
\item[(4)] $\gamma = 0$.
\end{enumerate}
\begin{proof}
Indeed, $\alpha = 1$ gives an $A_j\subseteq L$ belonging to $\mathcal{A}$. The first two statements are based on the fact that $D[S,A_j]$ is a complete bipartite digraph. Let $w\in A_j$, $y\in S$. Then there is a $u\in L$ such that $\{y,u\}\in \mathcal{A}$. Since $L$ is independent, the greedy requirement between $u$ and $A_j$ implies that $u$ (and its class $\{y,u\}$) precedes $A_j$ in $\mathcal{A}$. Then there arcs between $w$ and the class $\{y,u\}$, it should be $wy$ and $yw$, and thus  $D[S,A_j]$ is a complete bipartite digraph.

In order to prove (1) suppose, for a contradiction, that $B\subset S$ for $B\in \mathcal{B}$. Take any element $w \in A_j$. This implies $w\in L$ which by the definition of $L$ gives $\{w\}\in \mathcal{B}$, too, and thus there must be a non-arc between $w$ and $B$, a contradiction.

To prove (2) suppose, on the contrary, that $B\in \mathcal{B}$, $B\subseteq K\cup S$, and $B\cap K = \{v\}$. Since $\{w\}\in\mathcal{B}$ for all $w\in A_j$, there is a non-arc from $w$ to $B$, therefore $vw$ and $wv$ are arcs. Consider $\{v\}\in \mathcal{A}$ and $A_j$. There should be arcs $vw_1w_2v$, $w_1,w_2\in A_j$, a contradiction.

To prove (3) suppose $A_i \cap T = \{x\}$ and $(A_i\setminus \{x\})\subseteq L$. Notice that $i<j$ otherwise $u\in A_i\cap L$ would violate the greedy requirement between $u$ and $A_j$ in $\mathcal{A}$. Then there is a digon from $w\in A_j$ to $x$. By definition of $T$ there is a $v\in K$ such that $\{v,x\}\in \mathcal{B}$. Consider $\{w\}$ and $\{v,x\}$ in $\mathcal{B}$, $vw,wv \in F(D^c)$ follows (for every $w\in A_j$). Then the greedy requirement on $D$ is violated between the classs $A_j$ and $\{v\}\in\mathcal{A}$.

Note that (4) is a particular case of (3).
\end{proof}
\end{claim}

Similar to Claim \ref{claim7}, if $\beta=1$, then $\gamma = 0$. Conversely, $\gamma = 1$ implies $\alpha = \beta = 0$, hence $s \leq 1$ and we are done. From now on, we suppose that $\gamma = 0$, that is, $|S \cap T| = 0$, and then $s \leq 4$ and $dG(D)+dG(D^c)\leq \frac{5n+4}{4}.$ 

In the sequel, we will prove that for $n\geq 10$, $s\le2$ and in this case we have that $dG(D)+dG(D^c)\leq \frac{5n+2}{4}$. 

Since $s\leq 2(\alpha + \beta)-\epsilon$, if $\alpha + \beta \leq 1$ or $\epsilon \geq 2$ it follows that $s \leq 2$.

Assume that  $\alpha = \beta = 1$ and  $ \epsilon \leq 1$. In this case there exists a class $A'\in \mathcal{A}$, $A'\subseteq L$ (naturally, it is disjoint from $L_1$), and there exists a class $B'\in \mathcal{B}$, $B'\subseteq K$ (and $B'\cap K_1=\emptyset$). 
We claim that there is no class $A\in\mathcal{A}$ contained in $L\cup T$, other than $A'$. 
Claim \ref{claim7} implies that such a class $A$ intersects  both $L$ and $T$ in at least two vertices. If such an $A$ exists then $\epsilon_a \geq 1$ in Equation (\ref{eq4}). Also, $A$ should be counted twice on the left-hand-side of Equation (\ref{eq6}), implying $\epsilon_\alpha \geq 1$. Contradicting  $\epsilon \leq 1$.
Similarly, there is no second $B$-class in $K\cup S$.  

Let $W=V(D)\setminus (K \cup L \cup S \cup T )$, $|W|=\epsilon_\gamma$.
Consider the case $W=\emptyset$. Then there is no $A$-class covering the points of $T$, so $T$ should be empty. Similarly, $S=\emptyset$ follows. Then $V(D)= K\cup I$, hence $dG(D)+dG(D^c)\leq n+2$ and we are done.

Let $W\not= \emptyset$, since $ \epsilon \leq 1$ it follows that $|W| = 1$ and $\epsilon_a = \epsilon_b = \epsilon_\alpha = \epsilon_\beta = 0$. Let $A''$ be the $A$-class covering $W$. There are no more $\mathcal{A}$-classs in $T\cup (L \setminus L_1 ) \cup W$ so $|\mathcal{A}| = |K| + |S| + 2$. Similarly, $W\in B''\in \mathcal{B}$ and $|B| = |L| + |T| + 2$ giving $dG(D)+dG(D^c)\leq n+3$. Since $n + 3 \leq (5n + 2)/4$ we are done for $n \geq 10$.
\\

To finish, let $n\leq 9$. For these cases $dG(D)+dG(D^c)\leq \frac{5n+4}{4}$, implies $dG(D)+dG(D^c)\leq n+3$. We claim that if $dG(D)+dG(D^c)= n+3$ then $n=9$.
To prove the previous take the addition of the following seven pairwise disjoint sets: 
\[n \geq |A'|+|B'|+|A''\setminus E|+|K_1|+|B''\setminus E|+|L_1|+|E|.\]
Here $|A'|\geq 2$, $|B'|\geq 2$, $|E|=1$. It is easy to see that $|A''\setminus E| + |K_1 |\geq 2$ and $|B''\setminus E|+|L_1|\geq 2$. Indeed, $K_1 = \emptyset$ implies $T=\emptyset$ and $A''\subseteq L \cup E$. Since $E\notin S$ we get $|A''|\geq 3$.

\end{proof}


\section*{Acknowledgments}
%
%
%
Partially supported by PAPIIT-M\'exico: IN108121, CONACyT-M\'exico 282280, A1-S-12891, 47510664.

\end{document}